\documentclass[10pt]{amsart}
\usepackage[english]{babel}
\usepackage[DIV10, headinclude, BCOR 10mm]{typearea}
\usepackage{amssymb}
\usepackage[initials]{amsrefs}
\usepackage{mathrsfs} 			
\usepackage{enumitem}
\usepackage{color}
\usepackage{comment}

\usepackage[all]{xy}
\SelectTips{cm}{}

\newcommand{\bbN}{{\mathbb N}}

\newcommand{\bbZ}{{\mathbb Z}}
\newcommand{\bbC}{{\mathbb C}}

\newcommand{\bbF}{{\mathbb F}}

\newcommand{\calW}{\mathcal{W}}
\newcommand{\calU}{\mathcal{U}}
\newcommand{\calV}{\mathcal{V}}

\newcommand{\calO}{\mathcal{O}}
\newcommand{\expected}{\operatorname{\mathbb{E}}}

\newcommand{\bs}{\backslash}

\newcommand{\id}{\operatorname{id}}
\newcommand{\im}{\operatorname{im}}

\newcommand{\pr}{\operatorname{pr}}

\newcommand{\vol}{\operatorname{vol}}

\newcommand{\ricci}{\operatorname{Ricci}}
\newcommand{\const}{\operatorname{const}}

\newcommand{\tors}{\operatorname{tors}}
\newcommand{\coker}{\operatorname{coker}}

\newcommand{\rank}{\operatorname{rk}}
\newcommand{\spt}{\operatorname{spt}}
\newcommand{\ind}{\operatorname{ind}}

\newcommand{\proj}{\operatorname{proj}}

\newcommand{\mult}{\operatorname{mult}}

\newcommand{\abs}[1]{{\left\lvert #1\right\rvert}}

\newcommand{\norm}[1]{{\bigl\lVert #1\bigr\rVert}}

\newtheorem{theorem}{Theorem}[section]
\newtheorem{lemma}[theorem]{Lemma}

\newtheorem{corollary}[theorem]{Corollary}

\theoremstyle{definition}

\newtheorem{defn}[theorem]{Definition}

\newtheorem{remark}[theorem]{Remark}

\numberwithin{equation}{section}

\begin{document}
\title{Volume and homology growth of aspherical manifolds}
\author{Roman Sauer}
\address{Karlsruhe Institute of Technology, Karlsruhe, Germany}
\email{roman.sauer@kit.edu}
\thanks{The author gratefully acknowledges support by the DFG grant 1661/3-1. I thank Jonathan Pfaff very much 
for pointing out a mistake in an earlier version where the result were stated incorrectly for 
arbitrary coefficient systems. I also thank the referee, who spotted the same mistake, for an extremely helpful report.}

\begin{abstract}
1) We provide upper bounds on the size of the homology 
of a closed aspherical Riemannian manifold that only depend on the systole and the volume of balls. 2) We show that 
linear growth of mod~$p$ Betti numbers or exponential 
growth of torsion homology imply that a closed aspherical 
manifold is ``large''.  
\end{abstract}

\maketitle

%\tableofcontents

\section{Introduction and statement of results}

\subsection{Introduction} Let $M$ be a manifold whose fundamental group $\Gamma=\pi_1(M)$ 
is \emph{residually finite}. That is, $\Gamma$ possesses a decreasing sequence -- called 
a \emph{residual chain} -- of
normal subgroups $\Gamma_i<\Gamma$ of finite index whose intersection
is trivial. By covering theory there is an associated sequence of finite regular 
coverings $\ldots\to M_2\to M_1\to M$ of $M$ such that $\pi_1(M_i)\cong\Gamma_i$ and 
$\deg(M_i\to M)=[\Gamma:\Gamma_i]$, which we call 
a \emph{residual tower of finite covers}. A basic question is: 
\begin{center}
How does the size of the homology of $M_i$ grow as $i\to\infty$?
\end{center}
The growth behaviour of the first homology was connected to largeness of groups by 
Lackenby~\cite{lackenby-large,lackenby-tau}. It is also related to the cost~\cite{abert+nikolov}. 
Number theoretic connections of homology growth in the context of arithmetic locally symmetric spaces 
are discussed in~\cite{bergeron+venkatesh}. 

What do we mean by size? 
If we measure size by Betti numbers $b_k(M_i)=\rank_\bbZ H_k(M_i;\bbZ)$, 
there is a general answer: the limit of $b_k(M_i)/[\Gamma:\Gamma_i]$ is the $k$-th $\ell^2$-Betti number of $M$ by a result of L\"uck~\cite{lueck}.  
If we measure size by mod~$p$ Betti numbers or in terms of the cardinality 
of the torsion subgroups $\tors H_k(M_i;\bbZ)\subset H_k(M_i;\bbZ)$, 
no general answer is available. 

We shall consider throughout this paper the case that $M$ is a closed \emph{aspherical} manifold, that is, 
its universal cover~$\widetilde M$ is contractible. The manifold $M$ is aspherical 
if and only if $M$ is a model for the classifying space of its fundamental group. 

Our aim is to establish upper bounds for the homology and the homology growth 
of aspherical manifolds. Our gap theorem (Theorem~\ref{thm: gap theorem}) 
shows that the volume of a 
closed aspherical Riemannian manifold whose growth of torsion homology 
is exponential and whose Ricci curvature is bounded from below by $-1$ is greater 
than a positive universal constant. This establishes a link between the homology growth of aspherical manifolds and 
largeness of Riemannian manifolds in the sense of Gromov.  
Rather than being a precise notion, 
\emph{largeness} stands here 
for a variety of phenomena in Riemannian geometry~\cite{gromov-large}. For instance, 
we would regard an aspherical manifold as large if it has non-zero minimal volume. 

Important examples 
of aspherical manifolds are locally symmetric spaces of non-compact type. 
Bergeron and Venkatesh~\cite{bergeron+venkatesh} develop a detailed, yet largely conjectural, picture for the (torsion) homology growth of arithmetic locally symmetric spaces of non-compact type. 
By relating the growth of the torsion homology of $M_i$ to the 
analytic $\ell^2$-torsion, they are able to compute precisely the growth of torsion homology 
for special coefficient systems. 

\subsection{Statement of results}

The \emph{systole} of a Riemannian manifold is the minimal length of a non-contractible 
loop. The first result, which is based on the remarkable work of Guth~\cite{guth}, 
is for the homology of one manifold at a time but has an immediate consequence for the homology growth 
(Corollary~\ref{cor: residual chain}). 

In the sequel upper bounds for the $\bbF_p$-Betti numbers are 
formulated where $p$ stands for an arbitrary prime. By the universal coefficient 
theorem the Betti number in degree~$k$ is bounded above by the $\bbF_p$-Betti number 
in degree~$k$ for any $p$. So each theorem below also yields a bound on the Betti numbers. 
The theorems of this section can be easily extended from constant to unitary coefficients but we refrain from doing so to keep the exposition short and easier to read.

\begin{theorem}\label{thm: main thm volume}
	For every $n\in\bbN$ and $V_0>0$ 
	there exists a constant $\const(n,V_0)>0$ with 
	the following property: 
	Let $M$ be an $n$-dimensional  
	closed aspherical Riemannian manifold such that 
	every $1$-ball of $M$ has volume at most $V_0$ 
	and the systole of $M$ is at least~$1$. Then for every 
	$k\in\bbN$ 
	\begin{align*}
		\dim_{\bbF_p} H_k(M;\bbF_p) &< C(n,V_0)\vol(M)\text{ and }\\
		\log\bigl(|\tors H_k(M;\bbZ)|\bigr)&< C(n,V_0)\vol(M).
	\end{align*}
\end{theorem}

\begin{remark}
Gromov~\cite{ballmann+gromov+schroeder} proved 
that all Betti numbers of a real-analytic, closed or finite volume Riemannian manifold $M$ whose sectional 
curvature is between $-1$ and $0$ are bounded by $C(n)\vol(M)$. In contrast, 
our asssumptions are curvature-free but require 
a condition on the systole. For estimates of the torsion homology for non-compact arithmetic 
locally symmetric manifolds we refer to~\cites{gelander, emery}. 
\end{remark}

\begin{corollary}\label{cor: residual chain}
		For every $n\in\bbN$ and $V_0>0$ there is $C(n,V_0)>0$ with the following property: 
		Let $M$ be an $n$-dimensional  
		closed connected aspherical Riemannian manifold such that 
		every $1$-ball of the universal cover $\widetilde M$ has volume at most $V_0$. 
		Assume that the fundamental group is residually finite, and let $(M_i)$ be a 
		residual tower of finite covers. Then for every~$k\in\bbN$
		\begin{align*}
				\limsup_{i\to\infty}\frac{\dim_{\bbF_p} 
				  H_k(M_i;\bbF_p)}{\deg(M_i\to M)}&< C(n,V_0)\vol(M)\text{ and }\\
			\limsup_{i\to\infty}\frac{\log\bigl(|\tors
			  H_k(M_i;\bbZ)|\bigr)}{\deg(M_i\to M)}&< C(n,V_0)\vol(M).
	    \end{align*}
\end{corollary}

\begin{proof}[Proof of corollary]
 It is clear that the systole of $M_i$ converges to $\infty$ as $i\to \infty$ in a residual 
tower of finite covers. Note also that the maximal volume of a $1$-ball of $M$ coincides 
with the maximal volume of a $1$-ball in the universal cover~$\widetilde M$ provided the systole 
is at least~$1$. Now apply Theorem~\ref{thm: main thm volume}. 
\end{proof}

\begin{remark}
There is a natural tension between the volume on $M$ and on $\widetilde M$ 
in the above statements: If one scales the metric of $M$ by a
factor $<1$, then the volume of $M$ decreases, but the curvature and
so the volume of balls in the universal covering increase. 

If the
Ricci curvature is $\ge -1$ (short for $\ge -g$ as quadratic
forms), the volume of $1$-balls in $M$ and $\widetilde M$ is bounded
from above by a positive constant only depending on the dimension
according to the Bishop-Gromov inequality. 
\end{remark}

Our next result exhibits a gap phenomenon for the homology growth 
under a lower Ricci
curvature bound. 

\begin{theorem}\label{thm: gap theorem}
For every $n\in\bbN$ there is a constant $\epsilon(n)>0$ with the following
property: 
Let $M$ be a closed connected aspherical $n$-dimensional Riemannian manifold~$M$ 
such that $\ricci(M)\ge -1$ and the volume of every $1$-ball in~$M$ is at most $\epsilon(n)$. 
		Assume that the fundamental group is residually finite, and let $(M_i)$ be a 
		residual tower of finite covers. Then for every~$k\in\bbN$ 
\begin{align*}
\lim_{i\to\infty}~\frac{\dim_{\bbF_p} H_k(M_i;\bbF_p)}{\deg(M_i\to M)}&=0\text{ and }\\
     \lim_{i\to\infty} \frac{\log\bigl(|\tors
    H_k(M_i;\bbZ)|\bigr)}{\deg(M_i\to M)}&=0.
 \end{align*}
\end{theorem}

Gromov
showed~\cite{gromov}*{Section~3.4} that for every dimension $n$ 
there is a constant $\epsilon(n)>0$ with the following property: 
Every closed $n$-dimensional Riemannian manifold~$M$ 
such that $\ricci(M)\ge -1$ and the volume of every $1$-ball in~$M$ is 
at most~$\epsilon(n)$ can be covered by open, amenable sets with multiplicity $\le n$. 
Here a subset $U$ of a topological space $X$ is called
\emph{amenable} if the image of the map $\pi_1(U;x)\to \pi_1(X;x)$ on 
fundamental groups is amenable for any base point $x\in U$. 
Hence Theorem~\ref{thm: gap theorem} is a direct
consequence of the following topological result.

\begin{theorem}\label{thm: amenable covers}
Let $M$ be a closed connected 
aspherical $n$-dimensional manifold $M$. Assume that $M$ is covered by open, 
	amenable sets such that every point is contained in no more than
        $n$ such subsets. 
		Assume that the fundamental group is residually finite, and let $(M_i)$ be a 
		residual tower of finite covers. Then for every~$k\in\bbN$
\begin{align*}
\lim_{i\to\infty}~\frac{\dim_{\bbF_p} H_k(M_i;\bbF_p)}{\deg(M_i\to M)}&=0\text{ and }\\
     \lim_{i\to\infty} \frac{\log\bigl(|\tors
    H_k(M_i;\bbZ)|\bigr)}{\deg(M_i\to M)}&=0.
 \end{align*}
\end{theorem}

\begin{remark}[Relation to Gromov's vanishing theorem]
      Theorems~\ref{thm: gap theorem} and~\ref{thm: amenable covers} remind of 
	  the \emph{isolation theorem} and the \emph{vanishing theorem} 
	  in Gromov's seminal work~\cite{gromov}*{0.5. and~3.1.}, which under the 
	  same assumptions conclude the vanishing of the simplicial volume. The formal 
	  deduction of Theorem~~\ref{thm: gap theorem} from Theorem~\ref{thm: amenable covers} 
	  corresponds to Gromov's deduction of the isolation theorem from the 
	  vanishing theorem. Gromov's proof of the vanishing theorem is based on 
	  bounded cohomology and I do not see how to deduce something for 
	  the integral homology with this method. We develop a different approach in 
	  Section~\ref{sec:amenable_covers}. 
 \end{remark}

\begin{remark}
Vanishing results for the homology growth of aspherical spaces whose 
fundamental groups contain an infinite, normal, elementary amenable subgroup 
are proved in~\cite{lueck-homology}. The methods there are completely different 
from ours. If a normal amenable subgroup of $\pi_1(M)$ arises as the fundamental group 
of the fiber of a fiber bundle $F\to M\to N$ of closed aspherical manifolds, 
the assumptions in the previous theorem are satisfied for $M$ according 
to~\cite{gromov}*{Corollaries~(2) on p.~41}. 
\end{remark}

\begin{remark}[Consequences for $\ell^2$-Betti numbers]\label{rem: consequences l2 betti}
L\"uck's approximation theorem~\cite{lueck} 
yields as a corollary of 
Theorems~\ref{thm: gap theorem} 
and~\ref{thm: amenable covers} that all $\ell^2$-Betti
numbers of $M$ vanish. This has been proved earlier by the
author~\cite{sauer-volume}*{Corollary of Theorem~B}. 
Similarly, in Corollary~\ref{cor: residual chain} all $\ell^2$-Betti numbers of $M$ are bounded by
$C(n,V_0)\vol(M)$, which generalizes~\cite{sauer-volume}*{Corollary of Theorem~A}
\end{remark}

\subsection{On the proofs} 
As explained before, Theorem~\ref{thm: gap theorem} is a consequence of 
Theorem~\ref{thm: amenable covers}. The proofs 
of Theorems~\ref{thm: main thm volume} and~\ref{thm: amenable covers} start by 
a reduction to the orientable case. If there was no $2$-torsion in the 
homology groups in question, the reduction would be just an easy transfer argument. Of course, 
we do not want to assume that, so the reduction argument requires more care. This is done 
in Section~\ref{sec: reduction}. We may henceforth assume that all manifolds are 
oriented. 

The broad theme of this paper is the relation between the homology growth on aspherical 
manifolds and volume. Since the volume of an oriented 
Riemannian manifold~$M$ is related to 
its homology through 
the volume form or its dual, the fundamental class, this suggests 
an important role of the fundamental class in our proofs. We first present an  
outline of the proof of Theorem~\ref{thm: amenable covers}. 

Let $(M_i)$ be a residual tower of finite coverings of the closed aspherical manifold $M$ in question which 
is associated to a residual chain $(\Gamma_i)$ of the fundamental group. 
The proof consists of two major steps. 

\begin{enumerate}
	\item Bound the integral complexity of the fundamental class 
	\[[M_i]\in H_n(M_i;\bbZ)\cong H_n(\Gamma_i;\bbZ).\] 
	By \emph{integral 
	complexity} we mean the minimal number $l\in \bbN$ such that 
	the fundamental class is represented as an integral linear combination 
	of $l$ singular simplices. 
	\item Bound the size of the homology (torsion and free part) of $M_i$, thus $\Gamma_i$, in arbitrary degrees in terms of the integral complexity of $[M_i]$. 	
\end{enumerate} 
The second step is dealt with in Section~\ref{sec:fundamental class}. 
The first step, 
on which we elaborate now, is done in Section~\ref{sec:amenable_covers}. 
A central object of the proof is the profinite topological $\Gamma$-space 
\begin{equation}\label{eq: projective limit}
	X:=\varprojlim \bigl(\Gamma/\Gamma_0\leftarrow \Gamma/\Gamma_1\leftarrow\Gamma/\Gamma_2\leftarrow\ldots\bigr). 
\end{equation}
The space $X$ is also a compact topological group; we endow $X$ with the normalized 
Haar measure~$\mu$. To consider the space $X$ is motivated by the work of Ab{\'e}rt 
and Nikolov~\cite{abert+nikolov}  
who relate the cost of the $\Gamma$-action on $X$ to the rank gradient of $(\Gamma_i)$. 
We take the cover of $M$ by amenable subsets that figures in the assumption of 
Theorem~\ref{thm: amenable covers} and produce from it -- by dynamical considerations -- a different, $\Gamma$-equivariant measurable cover 
of a different object, namely the space $X\times\widetilde M$ endowed with the diagonal 
$\Gamma$-action. Each set in this cover is a product of a measurable set in $X$ and an open set in $\widetilde M$. 
We modify the measurable cover a little (Lemma~\ref{lem: approximation by cylindrical}) 
so that the measurable set in $X$ is cylindrical with respect to 
the profinite topology. 

Next we convert this measurable cover to a cover of the diagonal $\Gamma$-space 
$\Gamma/\Gamma_i\times\widetilde M$ for sufficiently large $i\in\bbN$ -- see~\eqref{eq: cover from correspondence principle} for the definition 
of the associated cover of $\Gamma/\Gamma_i\times\widetilde M$. 
The nerve of the cover of $\Gamma/\Gamma_i\times\widetilde M$ is a $\Gamma$-space; we denote its orbit space, which is 
a $\Delta$-complex,  
by $S(i)$. With an open cover and an associated partition of unity comes a map from the space to the nerve; in our situation the nerve map 
is equivariant and we will study the induced map $f$ on orbit spaces: 
\[    \Gamma\bs \bigl(\Gamma/\Gamma_i\times \widetilde M\bigr)\cong  \Gamma_i\bs\widetilde M=\xymatrix{  M_i\ar[r]^f &
  S(i)\ar@/^1pc/[l]^g } , 
\]
From asphericity we conclude that there is a homotopy retract $g\circ f\simeq\id$. The construction of the 
measurable cover above has been set up in such a way that the number of $n$-simplices in $S(i)$ is of magnitude $o([\Gamma:\Gamma_i])$. 
Finally, the homotopy retract implies that the integral complexity of $[M_i]$ is  $o([\Gamma:\Gamma_i])$, concluding the argument for the first step. 

The proof of Theorem~\ref{thm: main thm volume} also uses nerves of covers (albeit 
very different ones). 
It is much shorter than the proof of Theorem~\ref{thm: amenable covers} 
since the assertion is a rather direct consequence of Sections~\ref{sec: reduction} 
and~\ref{sec:fundamental class} and the remarkable work of Guth~\cite{guth}.

\section{Reduction to the orientable case}\label{sec: reduction}
 
\subsection{Estimates by subgroups of index~$2$}

\begin{lemma}\label{lem: short exact}
	If $0\to A\to B\to C$ is an exact sequence of finitely generated abelian groups, 
	then $|\tors B|\le |\tors A|\cdot |\tors C|$. 
\end{lemma}

\begin{proof}
	The given exact sequence restricts to an exact sequence \[0\to \tors A\to \tors B\to \tors C\] 
	of finite abelian groups from which the assertion follows. 
\end{proof}

\begin{lemma}\label{lem: less exact}
	If $A\to B\to C$ is an exact sequence of finitely generated abelian groups and 
	$A$ is torsion, then $|\tors B|\le |\tors A|\cdot |\tors C|$. 
\end{lemma}

\begin{proof}
	Let $f$ be the map from $B$ to $C$. Then $0\to A/\ker(f)\to B\to C$ is exact and 
	$|\tors A|=|A|\ge |A/\ker(f)|$. Now apply the previous lemma. 
\end{proof}

The following two lemmas are only needed for trivial coefficients but their proofs do 
not become more complicated in the stated generality. 

\begin{lemma}\label{lem: reduction torsion}
	Let $\Gamma$ be a group of type $F_\infty$ and 	$\Lambda<\Gamma$ a subgroup 
	of index $2$. Let $W$ be a $\bbZ[\Gamma]$-module that is finitely generated free 
	as a $\bbZ$-module. 
	For $n\in\bbN$ we have 
	\[
		\log |\tors H_n(\Gamma, W)|\le \sum_{k\le n} \rank_\bbZ H_k(\Lambda; W)+2^n\sum_{k\le n}  \log |\tors H_k(\Lambda; W)|.
	\]
\end{lemma}

\begin{proof}
	The finiteness condition $F_\infty$ ensures that all homology group to be considered 
	are finitely generated. 
	We consider the Lyndon-Hochschild-Serre spectral sequence for the group extension 
	$0\to \Lambda\to\Gamma\to C_2\to 0$, where $C_2$ is the cyclic group of order~$2$: 
	\[
		E_{p,q}^2=H_p\bigl(C_2; H_q(\Lambda; W)\bigr)\Rightarrow H_{p+q}(\Gamma; W)
	\]
	For $p>0$ the group $E_{p,q}^2$ is torsion, hence also $E_{p,q}^\infty$. 
	The convergence of the spectral sequence means that there is an increasing filtration 
	$F^iH_k(\Gamma;W)$ of $H_k(\Gamma;W)$ with $F^{k}H_{k}(\Gamma; W)=H_k(\Gamma;W)$ 
	and $E_{0,k}^\infty\cong F^0H_k(\Gamma;W)$ 
	and short exact sequences 
	\[
		0\to F^iH_k(\Gamma;W)\to F^{i+1}H_k(\Gamma;W)\to E_{i+1, k-i-1}^\infty\to 0
	\]
	for $i=0,\ldots, k-1$. Let $N_j:=H_j(\Lambda;W)$. It follows inductively from these 
	short exact sequences and Lemma~\ref{lem: short exact} that 
	\begin{equation}\label{eq: from spectral sequence}
		\log |\tors H_{k}(\Gamma;W)|= \sum_{i=0}^{k} \log |\tors E_{i,k-i}^\infty|\le 
		\sum_{i=0}^k \log |\tors H_i(C_2; N_{k-i})|.
	\end{equation}
	 The torsion submodule of $N_j$ is a $\bbZ[C_2]$-submodule 
	of $N_j$. So there is an exact sequence of $\bbZ[C_2]$-modules 
	\begin{equation}\label{eq: short exact}
		0\to \tors N_j\to N_j\to F_j\to 0, 
	\end{equation}
    where $F_j$ is $\bbZ$-free. The $\bbZ[C_2]$-module $F_j$ decomposes as a 
	direct sum of indecomposable $\bbZ[C_2]$-modules: 
	\[
		F_j\cong F_j^1\oplus\ldots\oplus F_j^l
	\]
	By a (much more general) result of Diederichsen and Reiner~\cite{curtis+reiner}*{Theorem~(34.31) on p.~729} 
	there are only three 
	isomorphism types of 
	indecomposable $\bbZ[C_2]$-modules 
	which are finitely generated free as $\bbZ$-modules: the trivial module $\bbZ$, 
	the module $\bbZ^\mathrm{tw}$ with underlying $\bbZ$-module $\bbZ$ on which the generator of $C_2$ 
	acts by multiplication with  $-1$ and $\bbZ[C_2]$.Our use of  
	the Diederichsen-Reiner result is motivated by~\cite{emery}*{Remark~2}. 
	If $N\in \{\bbZ, \bbZ^\mathrm{tw}, \bbZ[C_2]\}$, 
	then by direct computation $|\tors H_j(C_2; N)|\le 2$. Since $l\le \rank_\bbZ(N_j)$, 
	we obtain that 
	\[
	      \log |\tors H_k(C_2; F_j)|\le \rank_\bbZ (N_j)\log(2)\le \rank_\bbZ(N_j).
	\]
	The $k$-th chain group in the standard bar resolution computing 
	$H_k(C_2;\tors N_j)$ is isomorphic to $(\tors N_j)^{2^k}$, which implies that  
	\[
		 \log |\tors H_k(C_2; \tors N_j)|\le 2^k\log |\tors N_j|.
	\]
	The short exact sequence~\eqref{eq: short exact} induces a long exact sequence 
	in group homology: 
	\[
		\ldots H_{k+1}(C_2;F_j)\to H_k(C_2; \tors N_j)\to H_k(C_2; N_j)\to H_k(C_2; F_j)\to H_{k-1}(C_2; \tors N_j)\ldots
	\]
	Since $H_k(C_2;\tors N_j)$ is torsion for $k\ge 0$, Lemma~\ref{lem: less exact} implies that 
	\begin{align*}
		\log |\tors H_k(C_2; N_j)| &\le \log |\tors H_k(C_2; \tors N_j)|+\log |\tors H_k(C_2; F_j)|\\
		&\le 2^k\log |\tors N_j|+\rank_\bbZ (N_j)
	\end{align*}
	for $k\ge 0$. Combined with~\eqref{eq: from spectral sequence} this concludes the proof. 
\end{proof}

\begin{lemma}\label{lem: orientable rank reduction}
	Let $\Gamma$ be a group of type $F_\infty$ and 	$\Lambda<\Gamma$ a subgroup 
	of index $2$. Let $W$ be a $\bbF_p[\Gamma]$-module that is finitely generated free 
	as a $\bbF_p$-module. 
	For $n\in\bbN$ we have 
	\[
		\dim_{\bbF_p} H_n(\Gamma, W)\le 2^{n}\sum_{k\le n} \dim_{\bbF_p} H_k(\Lambda; W).
	\]
\end{lemma}

\begin{proof}
	The proof is easier than the one of the previous lemma. The Lyndon-Hochschild-Serre 
	spectral sequence yields that 
	\[
		\dim_{\bbF_p} H_n(\Gamma; W)\le \sum_{i=0}^n \dim_{\bbF_p} H_i(C_2; H_{n-i}(\Lambda;W)).
	\]
	As the $k$-th chain group in the standard bar resolution computing 
	$H_i(C_2; H_{n-i}(\Lambda;W))$ is isomorphic to a sum of ${2^i}$ copies of 
	$H_{n-i}(\Lambda;W)$, we have 
	\[
		\dim_{\bbF_p}  H_i(C_2; H_{n-i}(\Lambda;W))\le 2^i\dim_{\bbF_p}H_{n-i}(\Lambda;W).\qedhere
	\]
\end{proof}

\subsection{Reduction of Theorems~\ref{thm: main thm volume}  and~\ref{thm: amenable covers} 
	to the orientable case}\label{sub: reduction}
	Let us assume that 
	Theorem~\ref{thm: amenable covers} holds for the orientable 
	case.
Let $M$ be a connected $n$-dimensional 
closed aspherical non-orientable manifold with fundamental group $\Gamma$. 
By asphericity the group homology of $\Gamma$ and the homology of $M$ 
are isomorphic. 
There is a unique connected $2$-sheeted cover 
$\bar M\to M$ such that $\bar M$ is orientable. The fundamental group 
$\Gamma'=\pi_1(\bar M)$ embeds into $\Gamma$ as a subgroup of index~$2$. 
If $(\Gamma_i)$ is a residual chain of $\Gamma$, then 
$(\Gamma'\cap \Gamma_i)$ is a residual chain of $\Gamma'$. Moreover, 
$\Gamma'\cap\Gamma_i$ is either equal to $\Gamma_i$ or a subgroup of 
$\Gamma_i$ of index~$2$. 

Let us assume that $M$ is covered by open, amenable subsets 
with multiplicity $\le n$. Then their preimages cover $\bar M$ with multiplicity $\le n$ 
and are amenable. 
By Theorem~\ref{thm: amenable covers} applied to $\bar M$ and 
the universal coefficient theorem we have 
\begin{align}\label{eq: orientable vanishing}
	\lim_{i\to\infty} \frac{1}{[\Gamma':\Gamma'\cap\Gamma_i]}\sum_{k=0}^n \rank_\bbZ H_k\bigl(\Gamma'\cap\Gamma_i; \bbZ\bigr)=&0,\\
	\lim_{i\to\infty}
	\frac{1}{[\Gamma':\Gamma'\cap\Gamma_i]}\sum_{k=0}^n \log|\tors H_k(\Gamma'\cap\Gamma_i; \bbZ)|&=0.\notag
\end{align}
Thus Lemma~\ref{lem: reduction torsion} yields 
\[
	\lim_{i\to\infty}
	\frac{1}{[\Gamma:\Gamma_i]}\log|\tors H_k(\Gamma_i; \bbZ)|=0
\]
for every $k$. 
From~\eqref{eq: orientable vanishing} with $\bbZ$ replaced by $\bbF_p$ 
and Lemma~\ref{lem: orientable rank reduction} we obtain that 
\[
	\lim_{i\to\infty} \frac{1}{[\Gamma:\Gamma_i]}\dim_{\bbF_p} H_k(\Gamma_i; \bbF_p)=0.
\]
The reduction for Theorem~\ref{thm: main thm volume} via 
Lemmas~\ref{lem: reduction torsion} and~\ref{lem: orientable rank reduction} is along similar lines.

\section{Bounds by the fundamental class} % (fold)
\label{sec:fundamental class}

The following lemma is due to Soul{\'e}~\cite{soule}*{Lemma~1} who gives credit to Gabber.

\begin{lemma}\label{lem: soule}
	Let $A$ and $B$ be finitely generated free $\bbZ$-modules. Let 
	$a_1,\ldots, a_n$ and $b_1,\ldots, b_m$ be $\bbZ$-bases of 
	$A$ and $B$, respectively. We endow  
	$B_\bbC=\bbC\otimes_\bbZ B$ with the Hilbert space 
	structure for which $b_1,\ldots, b_m$ is a Hilbert basis. 
	Let $f\colon A\to B$ be a homomorphism. Let $I\subset\{1,\ldots, n\}$ be a 
	subset such that $\{f(a_i)\mid i\in I\}$ is a basis of $\im(f)_\bbC$. 
	Then 
	\[
		|\tors\coker(f)|\le \prod_{i\in I} \norm{f(a_i)}.
	\]
\end{lemma}

\subsection{The torsion estimate} % (fold)
\label{sub:the_torsion_estimate}

This subsection is devoted to the proof of the following result.

\begin{theorem}\label{thm: estimate torsion homology from integral simplicial volume}
	Let $M$ be a closed $n$-dimensional oriented manifold. 
	If the fundamental class of $M$ is represented by an integral cycle 
	with $k$ singular $n$-simplices, then 
	\begin{equation*}
		\log\bigl(|\tors H_j(M;\bbZ)|\bigr)\le \log(n+1)\cdot\binom{n+1}{j+1}\cdot k. 
	\end{equation*}
	for every $j\ge 0$. 
\end{theorem}

\begin{defn}
	Suppose 
	$\tau=\sum_{i=1}^k a_i\cdot \sigma_i$ 
	is an integral singular cycle representing 
	the fundamental class 
	$[M]\in H_n(M; \bbZ)$. We define  
	$C_\ast^\tau(M)$ as 
	the subcomplex of the integral singular chain complex $C_\ast(M)$ 
	that is generated, in degree~$j$, by 
	all $j$-dimensional faces 
	of the $n$-simplices  $\sigma_1,\ldots, \sigma_k$. 
\end{defn}

\begin{lemma}\label{lem: poincare}
	Let $\tau$ be a representing cycle of the fundamental class of~$M$. 
	The restriction chain homomorphism 
		\[\hom_{\bbZ}(C_\ast(M), \bbZ)\to \hom_{\bbZ}(C_\ast^\tau(M),\bbZ)   \]
	induces injective maps in cohomology in all degrees. 
\end{lemma}

\begin{proof}
	Let $\tau=\sum_{i=1}^k a_i\cdot \sigma_i$. The chain homomorphism 
	\begin{gather*}
    	\_\cap\tau: \hom_{\bbZ}(C_{n-\ast}(M), \bbZ)\to C_\ast(M)\\
          \hom_{\bbZ}(C_{n-l}( M),\bbZ)\ni\phi\mapsto \sum_{i=1}^k a_i \cdot\phi(\sigma_i\lfloor_{n-l})\otimes\sigma_i\rfloor_{l}
    \end{gather*}
    given by the cap product 
	with $\tau$ is 
	a homology isomorphism by Poincar{\'e} duality. 
    Let $\psi$ be a singular cocycle 
	such that $[\psi]$ is in the kernel of the cohomology homomorphism induced by the restriction 
	homomorphism. Upon subtracting 
	a coboundary, we may assume that $\psi$ 
	vanishes on $C_j^\tau(M)$. By the formula above, $[\psi\cap\tau]=0$. Hence 
	$[\psi]=0$, so injectivity follows. 
\end{proof}

\begin{proof}[Proof of Theorem~\ref{thm: estimate torsion homology from integral simplicial volume}]
	Let 
	$\tau=\sum_{i=1}^k a_i\cdot \sigma_i$ with $\sigma_i\ne\sigma_j$ for $i\ne j$  be an integral singular cycle representing 
	the fundamental class of~$M$. 
	The set of $j$-dimensional faces  
	of the singular simplices $\sigma_1,\ldots, \sigma_k$ 
	is a $\bbZ$-basis 
	of $C_j^\tau(M)$. 
	This turns $C_\ast^\tau(M)$ into a based 
	$\bbZ$-chain complex. 
	
    For every based finitely generated free $\bbZ$-module we consider the Hilbert 
	space structure on its complexification for which the $\bbZ$-basis becomes a 
	Hilbert basis. 
    The norm of the image of every basis element under the differential 
	    $\bbC\otimes\partial_{j}: \bbC\otimes_\bbZ C_{j}^\tau(M)\to \bbC\otimes_\bbZ C_{j-1}^\tau(M)$ is at most $(j+1)$. 
    Since an $n$-simplex has $\binom{n+1}{j}$-many $(j-1)$-dimensional 
    faces, we have 
    \[
    	\rank_\bbZ\bigl(H_{j-1}(C_\ast^\tau(M))\bigr)
    	\le\rank_\bbZ\bigl( C_{j-1}^\tau( M)\bigr)\le \binom{n+1}{j}\cdot k.
    \]
   By Lemma~\ref{lem: soule} we get that 
    \begin{equation*}
          \log\bigl| \tors \bigl(H_{j-1}(C_\ast^\tau(M)))\bigr| \le 
                             \log\bigl| \tors \bigl(\coker(\partial_{j})\bigr)\bigr|   	
                          \le \log(j+1)\binom{n+1}{j}\cdot k.
    \end{equation*}    
	The universal coefficient theorem implies 
	\[
		\log\bigl|\tors H^j\bigl (\hom_{\bbZ}(C_\ast^\tau(M),\bbZ)\bigr )\bigr|\le \log(j+1)\binom{n+1}{j}\cdot k.
	\]
   By Lemma~\ref{lem: poincare} 
    $H^j(M;\bbZ)$ embeds into $H^j(\hom_\bbZ(C_\ast^\tau(M),\bbZ))$, and 
	this yields an embedding of the corresponding torsion subgroups. 
	By Poincar{\'e} duality we conclude that
    \[
        \log\bigl(\tors\bigl(H_{j}(M;\bbZ)\bigr)\bigr)=    	
\log\bigl(\tors\bigl(H^{n-j}(M;\bbZ)\bigr)\bigr)\le \log(n+1)\cdot \binom{n+1}{j+1}\cdot k. \qedhere
    \]
\end{proof}

\subsection{The rank estimate} % (fold)
\label{sub:the_rank_estimate}

The corresponding result for Betti numbers is stated 
in~\citelist{\cite{gromov-book}*{p.~301 and p.~307} \cite{lueck}*{Example~14.28}} with a 
different constant. For convenience we provide a quick proof.

\begin{theorem}\label{thm: estimate rank homology from integral simplicial volume}
	Let $M$ be a closed $n$-dimensional oriented manifold. 
	If the fundamental class of $M$ is represented by an integral cycle 
	with $k$ singular $n$-simplices, then 
	\begin{equation*}
		\dim_{\bbF_p} \bigl(H_j(M;\bbF_p)\bigr)\le \binom{n+1}{j}\cdot k. 
	\end{equation*}
	for every $j\ge 0$. 
\end{theorem}

\begin{proof}
	    We retain the notation from the previous subsection. The estimate 
	    \[
	    	\dim_{\bbF_p}\bigl(H^j\bigl(\hom_{\bbZ}(C_\ast^\tau(M), \bbZ)\bigr)\bigr)\le \binom{n+1}{j+1}\cdot k
	    \]
		holds 
		since $\dim_{\bbF_p}\bigl(\hom_{\bbF_p}(C_\ast^\tau(M), \bbZ)\bigr)$ 
		satisfies the same upper bound. 
		The $\bbF_p$-analog of Lemma~\ref{lem: poincare} holds true 
		by the same proof. 
		Hence $H^j(M;\bbZ)$ injects into 
		$H^j\bigl(\hom_{\bbZ}(C_\ast^\tau(M), \bbZ)\bigr)$ and so 
		\[
			\dim_{\bbF_p}\bigl(H^j(M;\bbZ)\bigr)\le \binom{n+1}{j+1}\cdot k. 
		\]
		The statement now follows from Poincar{\'e} duality. 
\end{proof}

\section{Proof of Theorem~\ref{thm: main thm volume}} 
\label{sec:volume}

The proof is based on the following result, which is implicitly contained 
in Guth's paper~\cite{guth}; a weaker version, where one replaces the bound on the 
volume of $1$-balls by the stronger assumption $\ricci(M)\ge -1$, can be extracted 
from Gromov's paper~\cite{gromov}. 

\begin{theorem}\label{thm: variation on guth}
	For every $n\in \bbN$ and $V_0>0$ there is a constant 
	$C(n,V_0)>0$ with the following property: 
	
	If $M$ is an $n$-dimensional  connected oriented  
		closed aspherical Riemannian manifold such that 
		every $1$-ball of $M$ has volume at most $V_0$ 
		and the systole of $M$ is at least~$1$, then there is 
		an integral cycle that represents the fundamental class of $M$ and 
		has at most $C(n,V_0)$ singular $n$-simplices. 
\end{theorem}
The 
following result about the simplicial volume, that is, 
about the complexity of real-valued cycles representing the fundamental class, 
is explicitly contained in Guth's work. 

\begin{theorem}[Guth]
	For every $n\in \bbN$ and $V_0>0$ there is a constant 
	$C(n,V_0)>0$ with the following property: 
	
	If $M$ is an $n$-dimensional  connected oriented  
		closed aspherical Riemannian manifold such that 
		every $1$-ball of $M$ has volume at most $V_0$ 
		and the systole of $M$ is at least~$1$, then 
		the simplicial volume is at most $C(n,V_0)\vol(M)$. 
\end{theorem}

\begin{proof}
	Just combine Lemmas~7 and~9 in~\cite{guth}. 
\end{proof}

Guth's proof of the above theorem also implies Theorem~\ref{thm: variation on guth} as we explain now. Guth considers a good cover of~$M$ in the 
sense of Gromov~\cite{guth}*{Section~1} and introduces a modification of the usual nerve construction of a cover, which 
is particularly adapted to covers by balls with varying radii, 
called the \emph{rectangular nerve}~\cite{guth}*{Section~3}. He shows that 
for the map $f$ from $M$ to the 
rectangular nerve associated to a suitable partition of unity 
the $n$-volume of $f(M)$ is bounded by $C(n)\vol(M)$~\cite{guth}*{Lemma~5} for some universal 
constant $C(n)>0$.  He concludes  
that the image of the fundamental class is homologous to a cycle in the 
$n$-skeleton whose number of singular $n$-simplices is bounded 
by $C'(n)\vol(M)$ for some universal constant $C'(n)>0$~\cite{guth}*{Proof of Lemma~9}. By asphericity 
the map $f$ has a left homotopy inverse~\cite{guth}*{Lemma~7}, and this implies Theorem~\ref{thm: variation on guth}. 

Finally, we deduce Theorem~\ref{thm: main thm volume} from Theorem~\ref{thm: variation on guth}. 

\begin{proof}[Proof of Theorem~\ref{thm: main thm volume}]
	Let $M$ satisfy the assumptions in Theorem~\ref{thm: main thm volume}. 
	According to Subsection~\ref{sub: reduction} we may assume that $M$ is oriented. 
	Theorem~\ref{thm: main thm volume}  is now a consequence of 
	Theorem~\ref{thm: variation on guth} and Theorem~\ref{thm: estimate torsion homology from integral simplicial volume} (for the statement about torsion) 
	or Theorem~\ref{thm: estimate rank homology from integral simplicial volume} 
	(for the statement about $\bbF_p$-dimension), respectively. 
\end{proof}

\section{Proof of Theorem~\ref{thm: amenable covers}}
\label{sec:amenable_covers}

\subsection{Setup} 
\label{sub:setup}

Throughout Section~\ref{sec:amenable_covers} we shall consider 
a connected closed $n$-dimensional oriented aspherical manifold~$M$ and 
adhere to the following notation. 
\begin{enumerate}
    \item $\calU$ denotes a cover of $M$ of multiplicity $\le n$ by open,
    amenable subsets. 
  \item $(\Gamma_i)_{i\in\bbN}$ is a residual chain of the fundamental
    group $\Gamma:=\pi_1(M)$.
  \item Let $(X,\mu)$ be the profinite topological group in~\eqref{eq: projective limit} 
  with the normalized Haar measure $\mu$. Let $\pi_i: X\to \Gamma/\Gamma_i$ be the 
  canonical projection. 
\end{enumerate}
  
The sets 
$\pi_i^{-1}({\gamma \Gamma_i})$ with $i\in\bbN$ and $\gamma\in\Gamma$ 
define a basis $\calO$ of the 
topology of $X$. A subset of $X$ 
is called \emph{cylindrical} if it is a finite union of 
elements in $\calO$. For every cylindrical set $A$ there is a smallest number $i\in\bbN$ such that $A$ is 
the $\pi_i$-preimage of a subset of $\Gamma/\Gamma_i$. We call this number the \emph{level} of $A$ and denote it 
by $l(A)$. Note that for every $i\ge l(A)$ 
	\begin{equation}\label{eq: level equality}
		\pi_i^{-1}\bigl(\pi_i(A)\bigr)=A.
	\end{equation}
We will not use the topological group structure on $X$ and only regard $X$ as 
a $\Gamma$-space with a $\Gamma$-invariant probability measure~$\mu$.

\subsection{Constructing measurable covers}
\label{sub:measurable covers}

The topological space $X\times\widetilde M$ is endowed with the
diagonal $\Gamma$-action. We say that a $\Gamma$-invariant collection $\calV$ of 
subsets of $X\times\widetilde M$ is 
a \emph{$\Gamma$-equivariant measurable cover} if the following 
conditions are satisfied. 
\begin{enumerate}
	\item An element of $\calV$ is a product of a measurable subset of $X$ and 
	an open subset of $\widetilde M$. 
	\item The union of elements of $\calV$ is $X\times \widetilde M$. 
	\item The $\Gamma$-set $\calV$ has only finitely many $\Gamma$-orbits. 
\end{enumerate}

We construct such covers by appealing to the generalized Rokhlin lemma
by Ornstein-Weiss which we recall first. 

A \emph{monotile} $T$ in a
group $\Lambda$ is a finite subset for which there is a subset $C$ such that 
$\{T\cdot c\mid c\in C\}$ is a partition of $\Lambda$. For finite
subsets $F, T\subset \Lambda$ we say that $T$ is \emph{$(F,
  \delta)$-invariant}, if the \emph{$F$-boundary} $\partial
T=\{\lambda\in T\mid \exists_{\gamma\in F}~ \gamma\lambda\not\in T\}$
satisfies 
\[                  \abs{\partial T}/\abs{T}<\delta. \]
Weiss showed that residually finite amenable groups possess
$(F,\delta)$-invariant monotiles for arbitrarily small~$\delta>0$~\cite{weiss}. The relevance of monotiles stems from the
following result~\cite{ornstein+weiss}.

\begin{theorem}[Ornstein-Weiss]\label{thm: ornstein-weiss}
	Let $T$ be a monotile in an amenable group $\Lambda$. Let $\Lambda$ act freely and $\nu$-preservingly 
	on a standard probability space $(Y,\nu)$. For every $\epsilon>0$ there is a measurable subset 
	$A\subset Y$ satisfying: 
	\begin{enumerate}
		\item For $\lambda,\lambda'\in T$ with $\lambda\ne\lambda'$ the sets $\lambda A$ and 
	$\lambda' A$ are disjoint. 
		\item $\nu\bigl(\bigcup_{\lambda\in T}\lambda A\bigr)>1-\epsilon$.
	\end{enumerate}
\end{theorem}

We may assume that the elements $U_i\subset M$ of $\calU=\{U_1,\ldots,U_m\}$ are
connected. 
By taking the connected components we possibly increase $m$ but we do not increase the 
multiplicity. Define 
\[
	\Lambda_i:= \im\bigl(\pi_1(U_i)\to \pi_1(M)\bigr)\subset \Gamma.
\]
This involves the choice of a base point in each $U_i$ and paths from them to the base point of $M$ which is not relevant for our discussion. By assumption, each $\Lambda_i$ 
is amenable. Since $\Gamma$ is residually finite, each $\Lambda_i$ is
residually finite. 
Let $\bar U_i$ be the regular covering of $U_i$ associated to the 
kernel of the homomorphism $\pi_1(U_i)\to \pi_1(M)$. Let $\pr: \widetilde M\to M$ 
denote the universal covering projection. 
The group $\Lambda_i$ acts on $\bar U_i$ by deck transformations. 
By covering theory we may and will choose a lift of the map $\bar U_i\to U_i\to M$ to $\widetilde M$. This lift 
yields a homeomorphism 
\[
	\Gamma\times_{\Lambda_i} \bar U_i\cong \pr^{-1}(U_i)
\]
of coverings of $U_i$ which allows us to regard $\bar U_i$ as a subset of $\widetilde M$. 
We choose an open subset $K_i\subset \bar U_i$ which is relatively compact in~$\widetilde M$ such 
that $\bar U_i=\Lambda_i\cdot K_i$ for every $i\in\{1,\ldots,m\}$. 
Properness of the $\Gamma$-action allows us to fix a finite subset $F\subset\Gamma$ with the property   
\begin{equation}\label{eq: choice of F}
	\gamma K_r\cap K_s\ne \emptyset\Rightarrow\gamma\in F 
\end{equation}
for all $r,s\in\{1,\ldots, m\}$. 
Our following construction of $\Gamma$-equivariant measurable covers depends on
a parameter $\delta>0$ (and on many other choices but they are
less relevant). 
Later in the proof we consider a sequence of 
$\Gamma$-equivariant measurable covers for $\delta\to 0$. \medskip\\
Let us fix a parameter $\delta>0$. 
For every $i\in\{1,\ldots, m\}$ we choose a 
monotile $T_i\subset \Lambda_i$ of the residually finite amenable group $\Lambda_i$ that is 
\emph{$(F\cap\Lambda_i,\delta)$-invariant}. Since each
$T_i^{-1}K_i\cup K_i$ is relatively compact, there is a finite subset
$E\subset\Gamma$ with the property 
\begin{equation}\label{eq: the set E}
     \bigl(T_r^{-1}K_r\cup K_r\bigr)\cap\gamma\bigl(T_s^{-1}K_s\cup
     K_s)\ne\emptyset\Rightarrow \gamma\in E
\end{equation}
for all $r,s\in\{1,\ldots, m\}$. Set 
\begin{equation}\label{eq: choice of epsilon}
	\epsilon:= 2^{-(n+1)}\abs{E}^{-(n+1)}m^{-1}\delta. 
\end{equation}
By Theorem~\ref{thm: ornstein-weiss} there are 
measurable subsets $A_i\subset X$ such that, for fixed $i$, the sets $\lambda A_i$, $\lambda\in T_i$, 
are disjoint and $R_i=X\bs T_iA_i$ has $\mu$-measure at most~$\epsilon$. 
\begin{defn}\label{def: measurable cover}
	\begin{align*}
        \calW_\delta &:=\bigl\{\gamma A_i\times \gamma T_i^{-1}K_i\mid i\in\{1,\ldots, m\},\gamma\in\Gamma\bigr\}\\
		\calV_\delta &:=\bigl\{\gamma A_i\times \gamma
		        T_i^{-1}K_i\mid i\in\{1,\ldots,m\}, \gamma\in \Gamma\bigr\}\cup
		        \bigl\{\gamma R_i \times\gamma K_i\mid i\in\{1,\ldots,m\},
		        \gamma\in \Gamma\bigr\}\\
				\calW^0_\delta &:=\bigl\{A_i\times T_i^{-1}K_i\mid i\in\{1,\ldots, m\}\bigr\}\\
				\calV^0_\delta &:=\bigl\{A_i\times T_i^{-1}K_i\mid i\in\{1,\ldots, m\}\bigr\}\cup \bigl\{ R_i\times K_i\mid i\in\{1,\ldots,m\}\bigr\}
	\end{align*}
\end{defn}

The subsets $\calV^0_\delta$ and $\calW^0_\delta$ are $\Gamma$-transversals 
of $\calV_\delta$ and $\calW_\delta$, respectively. 
Keeping in mind that $\Gamma K_i=\widetilde{M}$ one immediately verifies: 
\begin{lemma} $\calV_\delta$ is a $\Gamma$-equivariant
  measurable cover. 
\end{lemma}

\begin{remark}
	In~\cite{sauer-volume} we proved that the $\ell^2$-Betti numbers of $M$ vanish under the assumptions of 
	Theorem~\ref{thm: amenable covers}; this turns out be  
	a corollary to Theorem~\ref{thm: amenable covers} provided the fundamental group is 
	residually finite (cf.~Remark~\ref{rem: consequences l2 betti}). 
	In~\cite{sauer-volume} we use a similar construction of measurable covers 
	but we do not use the topological structure on $X$ and 
	the methods therein are   
	not suited for the consideration of mod $p$ or torsion homology. 
\end{remark}

\subsection{Managing expectations} % (fold)
\label{sub:small_average_number_of_k_}
A 
$j$-tuple whose entries lie in a set $Y$ and are pairwise distinct is called 
a \emph{$j$-configuration} in $Y$. 
Let $(s_1,\ldots,s_j)\in \{1,\ldots,m\}^j$. 
A $j$-configuration 
in $\calV^\delta$ has \emph{type} $(s_1,\ldots,s_j)$ if its $l$-th entry is 
$\gamma A_{s_l}\times \gamma T_{s_l}^{-1}K_{s_l}$ or 
$\gamma R_{s_l}\times\gamma K_{s_l}$ for some $\gamma\in\Gamma$. 
A $j$-configuration in $\calW^\delta$ has \emph{type} 
$(s_1,\ldots,s_j)$ if its $l$-th entry is 
$\gamma A_{s_l}\times \gamma T_{s_l}^{-1}K_{s_l}$ for some $\gamma\in \Gamma$.

\begin{defn}\label{def: mult functions}
	Let $\mult^{\calV_\delta}(j):X\to \bbN$ be the random variable whose value at $x\in X$ 
	is the number of $j$-configurations in $\calV_\delta$ 
    with the property that the common intersection of all sets in the configuration with the set $\{x\}\times\widetilde M$ 
	is non-empty and the first set of the configuration is from $\calV^0_\delta$. 
	Correspondingly, $\mult^{\calW_\delta}(j): X\to\bbN$ is defined. 
	For $(s_1,\ldots,s_j)\in\{1,\ldots, m\}^j$ 
	the random variables $\mult^{\calV_\delta}(j;s_1,\ldots,s_j)$ and 
	$\mult^{\calW_\delta}(j;s_1,\ldots,s_j)$ are similarly defined but only 
	$j$-configurations of type $(s_1,\ldots,s_j)$ are counted. 
\end{defn}

\begin{lemma}\label{lem: reduction to W cover}
	For every type $(s_1,\ldots,s_j)$, $2\le j\le n+1$, the expected values satisfy
	\[
		 \expected \mult^{\calV_\delta}(j;s_1,\ldots,s_j)\le \expected \mult^{\calW_\delta}(j;s_1,\ldots,s_j)+\delta.
	\]
\end{lemma}

\begin{proof}
	A crude estimate yields that 
	$\mult^{\calV_\delta}(j;s_1,\ldots,s_j)\le (2\abs{E})^j$, 
	where $E\subset\Gamma$ is the subset from~\eqref{eq: the set E}. 
	The support of 
	$\mult^{\calV_\delta}(j;s_1,\ldots,s_j)-\mult^{\calW_\delta}(j;s_1,\ldots,s_j)$ 
	is a subset of $E\cdot (R_1\cup \ldots\cup R_m)$. Hence the difference of 
	the corresponding expected values is at most 
	\[
		\mu\bigl( E\cdot (R_1\cup \ldots\cup R_m)  \bigr)(2\abs{E})^j
		\le 2^j\abs{E}^{j+1}m\epsilon\underset{\eqref{eq: choice of epsilon}}{\le}\delta.\qedhere
	\]
\end{proof}

\begin{lemma}\label{lem: going to zero for fixed type}
	Let $2\le j\le n+1$. If $j$ is greater than the multiplicity of $\calU$, then 
	\[
		\lim_{\delta\to 0} \expected \mult^{\calV_\delta}(j;s_1,\ldots,s_j)=\lim_{\delta\to 0} \expected \mult^{\calW_\delta}(j;s_1,\ldots,s_j)=0.
	\]
	for every type $(s_1,\ldots,s_j)$. 
\end{lemma}

\begin{proof}
	By Lemma~\ref{lem: reduction to W cover} it suffices to show the second equality. 
	The multiplicity of $\calU=\{U_1,\ldots, U_m\}$ is the same as the one of 
	$\{\pr^{-1}(U_1),\ldots, \pr^{-1}(U_m)\}$. Since $\Gamma K_i= \pr^{-1}(U_i)$, 
	$\mult^{\calW_\delta}(j;s_1,\ldots,s_j)$ is constant zero or the type 
	$(s_1,\ldots,s_j)$ has two identical components. For the remaining proof we thus may 
	assume that $s_1=s_2$. 
	
	Fix $\delta>0$. Let $C_l(x)$ be the $l$-configurations that contribute 
	to $\mult^{\calW_\delta}(l;s_1,\ldots, s_l)(x)$ for $l\le j$ and $x\in X$, so that 
	$\mult^{\calW_\delta}(l;s_1,\ldots, s_l)(x)=|C_l(x)|$. Note that the first set of any  
	$l$-configuration in $C_l(x)$ is $A_{s_1}\times T_{s_1}^{-1}K_{s_1}$. For $\lambda\in T_{s_1}^{-1}$ 
	we define $C_l^\lambda(x)$ as the subset of $C_l(x)$ that consists of $l$-configurations 
	with the additional property 
	that the common intersection of all sets of the configuration with $A_{s_1}\times \lambda K_{s_1}$ 
	is non-empty. Obviously, the union of all $C_l^\lambda(x)$ is $C_l(x)$ but it is not  necessarily a disjoint union.  
	At least we get that 
	\[
		\mult^{\calW_\delta}(l;s_1,\ldots, s_l)(x)=|C_l(x)|\le \sum_{\lambda\in T_{s_1}^{-1}}|C_l^\lambda(x)|.
	\]
	Let $3\le l\le j$ and $\lambda\in T_{s_1}^{-1}$. Let us take a look at a fiber of the 
	projection $C_l^\lambda(x)\to C_{l-1}^\lambda(x)$ that drops the last set of a configuration. The cardinality 
	of any fiber is bounded by the number of elements $\gamma\in \Gamma$ such that $x\in \gamma A_{s_l}$ and 
	$\lambda K_{s_1}\cap \gamma T_{s_l}^{-1}K_{s_l}\ne\emptyset$. By~\eqref{eq: choice of F} 
	the latter implies that $\gamma\in \lambda F T_{s_l}$. For every $\xi\in F$ the sets $\lambda \xi \theta A_{s_l}$, where 
	$\theta$ runs through $T_{s_l}$, are pairwise disjoint, so the given $x$ can only be in one of them. Thus 
	any fiber has at most $|F|$ elements. We obtain inductively that 
	\[
		\sum_{\lambda\in T_{s_1}^{-1}}|C_j^\lambda(x)|\le |F|^{j-2} \sum_{\lambda\in T_{s_1}^{-1}}|C_2^\lambda(x)|.
	\]
	Let $s:=s_1=s_2$. Every $2$-configuration in one of the sets $C_2^\lambda(x)$ has type $(s,s)$. 
	Define 
	\[ f(x):=| \bigl\{(\lambda, \theta, \gamma)\in T_s^{-1}\times T_s^{-1}\times \Gamma\bs\{e\}\mid x\in A_s\cap \gamma A_s~~\text{ and }~~
		\lambda K_s\cap\gamma \theta K_s\ne\emptyset\bigr\}|.
	\]
	For every $x\in X$ we have 
	\[\sum_{\lambda\in T_{s_1}^{-1}}|C_2^\lambda(x)|\le f(x).\] 
	Hence $\mult^{\calW_\delta}(j;s_1,\ldots, s_j)$ is dominated by $|F|^{j-2}f$ and so 
	\begin{equation}\label{eq: domination}
		\expected\mult^{\calW_\delta}(j;s_1,\ldots, s_j)\le |F|^{j-2}\expected f.
	\end{equation}
	We shall now use for the first time that the monotiles $T_i$ are 
	$(F\cap\Lambda_i, \delta)$-invariant. 
	Let $(\lambda, \theta,\gamma)\in T_s^{-1}\times T_s^{-1}\times \Gamma\bs\{e\}$ be a triple such that 
	$\lambda K_s\cap\gamma \theta K_s\ne\emptyset$. By~\eqref{eq: choice of F} 
	there is $\rho\in F$ with $\gamma=\lambda\rho\theta^{-1}$. 
	Suppose that $\theta^{-1}\not\in \partial T_s$. 
	Then $\rho\theta^{-1}\in T_s$. Because of $\gamma\ne e$ one has 
	 $\rho\theta^{-1}\ne \lambda^{-1}$, which yields $\lambda^{-1}A_s\cap\rho\theta^{-1}A_s=\emptyset$, thus $A_s\cap\gamma A_s=\emptyset$. 
	 Therefore, 
	 \[
	 	\expected f\le \sum_{(\lambda, \theta, \rho)}\mu\bigl(A_s\cap \lambda\rho\theta^{-1}A_s\bigr)=
		\sum_{(\lambda, \theta, \rho)}\mu\bigl(\lambda^{-1}A_s\cap\rho\theta^{-1}A_s\bigr)
	 \]
	where the summation runs over all triples in $T_s^{-1}\times (\partial T_s)^{-1}\times F$. 
	Since the sets $\lambda^{-1}A_s$, where $\lambda$ runs through $T_s^{-1}$, are 
	disjoint, we have 
	\[
		\sum_{(\lambda, \theta, \rho)}\mu\bigl(\lambda^{-1}A_s\cap\rho\theta^{-1}A_s\bigr)\le 
		\sum_{(\theta, \rho)}\mu\bigl(\rho\theta^{-1}A_s\bigr)=
		\sum_{(\theta, \rho)}\mu\bigl(A_s\bigr)
	\]
	where $(\theta, \rho)$ runs over $(\partial T_s)^{-1}\times F$. 
	Because of $\mu(A_s)\le 1/\abs{T_s}$ we obtain that 
	\[
		\expected f\le\abs{F}\abs{\partial T_s}\mu(A_s)\le \abs{F}\abs{\partial T_s}/\abs{T_s}\le |F|\delta. 
	\]
	With~\eqref{eq: domination} the proof is completed. 
\end{proof}

\begin{theorem}\label{thm: smallness of multiplicity}
	Let $2\le j\le n+1$. If $j$ is greater than the multiplicity of~$\calU$, then $\lim_{\delta\to 0} \expected \mult^{\calV_\delta}(j)=0$. 
	\end{theorem}

\begin{proof}
		The random variable $\mult^{\calV_\delta}(j)$ is the sum of the random variables  $\mult^{\calV_\delta}(j;s_1,\ldots,s_j)$, where $(s_1,\ldots,s_j)$ runs through all $m^j$ possible types. 
	 Thus the statement follows from Lemma~\ref{lem: going to zero for fixed type}. 
\end{proof}

\subsection{From measurable to open covers on $X\times\widetilde M$} 
\label{sub:approximating_measurable_covers}

Our next goal is to show that we can replace the measurable sets $A_i$ and
$R_i$ in the definition of $\calV_\delta$ 
by cylindrical (in particular, open) sets without losing
the property stated in Theorem~\ref{thm: smallness of multiplicity}. 

\begin{lemma}\label{lem: approximation by cylindrical}
Let $\delta>0$, and let $\calV_\delta$ be the measurable 
cover in Definition~\ref{def: measurable cover}. 
For every $\epsilon>0$ there are cylindrical subsets $A_i^c$ and $R_i^c$ of $X$ for every $i\in\{1,\ldots, m\}$ such that 
\begin{enumerate}
  \item $\calV_\delta^c:=\bigl\{\gamma A_i^c\times \gamma
        T_i^{-1}K_i\mid i\in\{1,\ldots,m\}, \gamma\in \Gamma\bigr\}\cup
        \bigl\{\gamma R_i ^c\times\gamma K_i\mid i\in\{1,\ldots,m\},
        \gamma\in \Gamma\bigr\}$ is a $\Gamma$-equivariant measurable
        cover. 
\item $\expected\mult^{\calV_\delta^c}(k)<\expected\mult^{\calV_\delta}(k)+\epsilon$ for $k\in
  \{1,\ldots, n+1\}$. 
\end{enumerate}
\end{lemma}

\begin{proof}
Let $\epsilon_0>0$. 
Being the Haar measure on the compact topological group~\eqref{eq: projective limit}, $\mu$ is regular. 
Hence there are open subsets
$A_i^o\supset A_i$ and $R_i^o\supset R_i$ in $X$ for every $i\in\{1,\ldots,
m\}$ such that $\mu(A_i^o\backslash A_i)<\epsilon_0$ and 
$\mu(R_i^o\backslash R_i)<\epsilon_0$. For any $i\in\{1,\ldots, m\}$ 
we can write $A_i^o$ and $R_i^o$ 
as increasing countable unions of cylindrical subsets:
\[
	A_i^o=\bigcup_{q\in\bbN}A_i^{(q)},\qquad R_i^o=\bigcup_{q\in\bbN}R_i^{(q)}
\]
since $\calO$ is a countable subbasis of the topology of~$X$. 
Let $\proj: X\times\widetilde M\to X\times_\Gamma\widetilde M$ be the quotient map for the diagonal action. 
Let 
\[
	S_q:=\bigcup_{i\in\{1,\ldots,m\}}\bigcup_{\gamma\in \Gamma}\bigl(\gamma A_i^{(q)}\times \gamma
	        T_i^{-1}K_i~\cup \gamma R_i^{(q)}\times \gamma
	        K_i\bigr).
\]
Since the sets in $\calV_\delta$ cover $X\times\widetilde M$, 
we have the following increasing unions: 
\[
	X\times\widetilde M=\bigcup_{q\in\bbN}S_q,\qquad X\times_\Gamma\widetilde M=\bigcup_{q\in\bbN}\proj(S_q)
\]
The orbit space $X\times_\Gamma\widetilde M$ 
is compact since $M$ and $X$ are compact. Further, $\proj$ is an open map. 
By compactness there is
$q_0\in\bbN$ such that $X\times_\Gamma\widetilde M=\proj(S_q)$ for every $q\ge q_0$. 
Hence $X\times\widetilde M=S_q$ for $q\ge q_0$ and 
\[
	\calW_q:=\bigl\{\gamma A_i^{(q)}\times \gamma
	        T_i^{-1}K_i\mid i\in\{1,\ldots,m\}, \gamma\in \Gamma\bigr\}\cup
	        \bigl\{\gamma R_i ^{(q)}\times\gamma K_i\mid i\in\{1,\ldots,m\},
	        \gamma\in \Gamma\bigr\}
\]
is a $\Gamma$-equivariant measurable cover for $q\ge q_0$. 
We claim that 
\begin{equation}\label{eq: measure approximation}
\expected\mult^{\calW_q}(k)<\expected\mult^{\calV_\delta}(k)+\epsilon
\end{equation}
for all $k\in\{1,\ldots,n+1\}$ provided $\epsilon_0>0$ is sufficiently small and $q\ge q_0$ is sufficiently 
large. For such $\epsilon_0$ and $q\ge q_0$ we then set $A_i^c:=A_i^{(q)}$ and 
$R_i^{c}:=R_i^{(q)}$ which finishes the proof. 

To show~\eqref{eq: measure approximation} it suffices to verify that 
for a fixed $k$-configuration $(s_1,\ldots, s_k)$ we have 
\begin{equation}\label{eq: measure approximation for configuration}
\expected\mult^{\calW_q}(k;s_1,\ldots,s_k)<\expected\mult^{\calV_\delta}(k;s_1,\ldots, s_k)+\epsilon
\end{equation}
provided $\epsilon_0>0$ is sufficiently small and $q\ge q_0$ is sufficiently 
large. To this end, we rewrite the expected values using Fubini's theorem. To formulate the result, we 
set: 
\[
		B_i(b) := \begin{cases}
		                A_i & \text{if $b=0$;}\\
						R_i & \text{if $b=1$;}
				\end{cases}~\text{ and }~
		C_i(b) := \begin{cases}
		                    T_i^{-1}K_i & \text{ if $b=0$;}\\
							K_i         & \text{ if $b=1$.}
				\end{cases}
\]
Similarly, we define $B_i^{(q)}(b)$ with $A_i$ and $R_i$ being replaced by $A_i^{(q)}$ and $R_i^{(q)}$, respectively. 
By Fubini's theorem we have 
\[
	\expected\mult^{\calV_\delta}(k;s_1,\ldots, s_k)=\sum_{\substack{\gamma_2,\ldots, \gamma_k\\b_1,\ldots,b_k}} \mu(B_{s_1}(b_1)\cap \gamma_2B_{s_2}(b_2)\cap\ldots\cap \gamma_kB_{s_k}(b_k))
\]
where the sum runs through $\gamma_2,\ldots\gamma_k\in\Gamma$ and $b_1,\ldots, b_k\in\{0,1\}$ with the property 
that $C_{s_1}(b_1)\cap \gamma_2 C_{s_2}(b_2)\cap\ldots \gamma_k C_{s_k}(b_k)\ne\emptyset$. For 
$\expected\mult^{\calW_q}(k;s_1,\ldots,s_k)$ we have a similar expression. Letting $\epsilon_0\to 0$ and $q\to\infty$, 
each summand $\mu(B_1^{(q)}(b_1)\cap \gamma_2B_2^{(q)}(b_2)\cap\ldots\cap \gamma_kB_k^{(q)}(b_k))$ tends to 
$\mu(B_1(b_1)\cap \gamma_2B_2(b_2)\cap\ldots\cap \gamma_kB_k(b_k))$. This 
yields~\eqref{eq: measure approximation for configuration} and thus~\eqref{eq: measure approximation}. 
\end{proof}

\subsection{The passage to open covers on $\Gamma/\Gamma_q\times\widetilde M$} % (fold)

Next we describe how we produce from a $\Gamma$-equivariant measurable cover on $X\times\widetilde M$ 
a sequence of $\Gamma$-equivariant open covers on $\Gamma/\Gamma_q\times \widetilde M$, 
indexed by $q\in\bbN$, which are compatible with respect to the natural projections 
$\Gamma/\Gamma_q\times\widetilde M\to \Gamma/\Gamma_{q-1}\times\widetilde M$. 

Let us fix $\delta>0$. 
By Lemma~\ref{lem: approximation by cylindrical}
one can replace the sets $A_i$ and $R_i$ in $\calV_\delta$ by cylindrical 
sets $A_i^c$ and $R_i^c$ such that 
\[
\calV_\delta^c=\bigl\{\gamma A_i^c\times \gamma
        T_i^{-1}K_i\mid i\in\{1,\ldots,m\}, \gamma\in \Gamma\bigr\}\cup
        \bigl\{\gamma R_i ^c\times\gamma K_i\mid i\in\{1,\ldots,m\},
        \gamma\in \Gamma\bigr\}
\]
is a $\Gamma$-equivariant measurable cover with 
\begin{equation}\label{eq: smaller than twice expectation}
  \expected\mult^{\calV_\delta^c}(n+1)<2\cdot \expected\mult^{\calV_\delta}(n+1). 
  \end{equation}
From $\calV_\delta^c$ we obtain the $\Gamma$-equivariant
open cover of $\Gamma/\Gamma_q\times\widetilde M$: 
\begin{multline}\label{eq: cover from correspondence principle}
\calV_\delta(q):=\bigl\{\{\pi_q(\gamma x)\}\times \gamma
        T_i^{-1}K_i\mid i\in\{1,\ldots,m\}, \gamma\in \Gamma, x\in A_i^c\bigr\}~\cup\\
        \bigl\{\{\pi_q(\gamma x)\}\times\gamma K_i\mid i\in\{1,\ldots,m\},
        \gamma\in \Gamma, x\in R_i^c\bigr\}
\end{multline}

We define the function $ \mult^{\calV_\delta(q)}(j): \Gamma/\Gamma_q\to \bbN$ similarly 
as $\mult^{\calV_\delta^c}(j)$. The value $\mult^{\calV_\delta(q)}(j)(y)$ 
at $y\in\Gamma/\Gamma_q$ 
is the number of $j$-configurations 
in $\calV_\delta(q)$ with the property 
that the common intersection of all sets in the configuration 
with $\{y\}\times\widetilde M$ is non-empty and the first set of the 
configuration is $\{y\}\times T_i^{-1}K_i$ or $\{y\}\times K_i$ for some $i\in\{1,\ldots, m\}$. 
We regard this function as a random variable on 
$\Gamma/\Gamma_i$ endowed with the equidistributed probability measure. Let 
\begin{equation}\label{eq: maximal level}
q_0:=\max\{ l(A_i^c), l(R_i^c)\mid i=1,\ldots, m)\}
\end{equation}
be the maximal level of the cylindrical sets appearing in
$\calV_\delta^c$.

\begin{lemma}\label{lem:technical lemma} For every $q\ge q_0$ and $j\in\bbN$ one has 
$\expected\mult^{\calV_\delta^c}(j)=\expected \mult^{\calV_\delta(q)}(j)$.
\end{lemma}

\begin{proof}
	For $x\in X$ and $y\in\Gamma/\Gamma_q$ 
	let $C_j(x)\subset (\calV_\delta^c)^j$ and $\tilde C_j(y)\subset\calV_\delta(q)^j$ 
	be the sets of $j$-configurations 
	that contribute to $\mult^{\calV_\delta^c}(j)(x)$ and $\mult^{\calV_\delta(q)}(j)(y)$, 
	respectively. So,  $\mult^{\calV_\delta^c}(j)(x)=|C(x)|$ and 
	$\mult^{\calV_\delta(q)}(j)(y)=|\tilde C(y)|$. 
	
	The map $\phi(x): C_j(x)\to\tilde{C}_j(\pi_q(x))$ is defined component-wise:  
	a $j$-configuration 
	whose $i$-th set is $A\times U$ (thus $x\in A$) is sent 
	to the $j$-configuration whose $i$-th set is $\{\pi_q(x)\}\times U$. 
    Injectivity of $\phi(x)$ is clear. Surjectivity is implied by 
    the equivalences  (cf.~\eqref{eq: level equality})
	\[\pi_q(x)\in \pi_q(A_i^c)\Leftrightarrow x\in A_i^c~\text{ and }~
	\pi_q(x)\in \pi_q(R_i^c)\Leftrightarrow x\in R_i^c\] 
	provided $q$ is 
	at least the level of $A_i^c$ and $R_i^c$. 
	Hence $\mult^{\calV_\delta^c}(j)(x)=\mult^{\calV_\delta(q)}(j)(\pi_q(x))$ 
	for every $x\in X$. That the pushforward of the measure $\mu$ under the map 
	$\pi_q$ is the equidistributed probability measure on $\Gamma/\Gamma_q$ finishes the 
	proof. 
\end{proof}

\subsection{Conclusion of the proof of Theorem~\ref{thm: amenable covers}}

The reader is referred 
to~\cite{hatcher}*{Chapter~ 2.1.} for a discussion of the upcoming 
notion of \emph{$\Delta$-complex}. Its historical name is \emph{semi-simplicial complex}. 
In brief, a $\Delta$-complex is like a simplicial complex
where one drops the requirement that a simplex is uniquely determined
by its vertices. A $1$-simplex in a $\Delta$-complex, for instance, might be a loop. 

\begin{lemma}\label{lem: final lemma for amenable thm}
For every $\delta>0$ there is $q_0\in\bbN$ with the following property. 
For  every $q\ge q_0$ there is a $\Delta$-complex $S(q)$ such that 
\begin{enumerate}
\item $S(q)$
has at most $2[\Gamma:\Gamma_q] \expected\mult^{\calV_\delta}(n+1)$ many $n$-simplices, 
\item and there is a homotopy retract 
\end{enumerate} 
\[   \xymatrix{   \Gamma_q\bs\widetilde M\ar[r]^f &
  S(q)\ar@/^1pc/[l]^g }~\text{ with } g\circ f\simeq \id_{\Gamma_q\bs\widetilde M}.
\]

\end{lemma}

\begin{proof}[Proof of Lemma~\ref{lem: final lemma for amenable thm}]

Let $\delta>0$. We consider the $\Gamma$-equivariant measurable 
cover $\calV_\delta^c$ from the previous subsection which we obtained 
by an application of Lemma~\ref{lem: approximation by cylindrical} to 
$\calV_\delta$. It satisfies~\eqref{eq: smaller than twice expectation}. 
Let $q_0$ be defined as in~\eqref{eq: maximal level}, as the maximal level of 
cylindrical subsets occuring in $\calV_\delta^c$. For every $q\in\bbN$ we 
obtain a $\Gamma$-equivariant cover $\calV_\delta(q)$ 
of $\Gamma/\Gamma_q\times\widetilde M$ 
from $\calV_\delta^c$ by~\eqref{eq: cover from correspondence principle}. 
By~\eqref{eq: smaller than twice expectation} 
and Lemma~\ref{lem:technical lemma} we have 
\begin{equation}\label{eq: endgame expectation}
		\expected \mult^{\calV_\delta(q)}(n+1)<2\expected\mult^{\calV_\delta}(n+1)~~\text{ for $q\ge q_0$.}
\end{equation}
Let $q\ge q_0$, and let $N(q)$ be the nerve of $\calV_\delta(q)$. 
Recall that the nerve of a
cover is the  simplicial complex whose vertices correspond to the
subsets of the cover such that $(k+1)$ subsets span a $k$-simplex if
they have a non-empty intersection. The $\Gamma$-action on $\calV_\delta(q)$
induces a simplicial $\Gamma$-action on $N(q)$.  Since the sets
in $\calV_\delta(q)$ are relatively compact, the 
$\Gamma$-action on $\widetilde M$ is proper, and $\Gamma$ is torsion-free, it follows that 
the
$\Gamma$-action on $N(q)$ is free. Moreover, if a simplex is invariant
under some $\gamma\in\Gamma$ as a set, then $\gamma=e$.  Thus $N(q)$
is a free $\Gamma$-CW-complex~\cite{dieck}*{Proposition~(1.15) on
  p.~101}. A free $\Gamma$-CW-complex is a Hausdorff space with a
$\Gamma$-action that is built inductively by attaching equivariant
cells $\Gamma\times D^i$ via equivariant attaching maps
(see~\cite{dieck}*{p.~98}). This will enable us below to construct equivariant 
maps with domain $N(q)$ by induction over skeleta. 
Let 
\[ S(q):=\Gamma\bs N(q).\]
Whilst $\Gamma$ acts simplicially on the simplicial complex $N(q)$,
the simplicial structure of $N(q)$ does not necessarily induce a
simplicial structure on $S(q)$. For instance, one might have a $1$-simplex
between a vertex $v$ and a vertex $\gamma v$ for $\gamma\in\Gamma$
which yields a loop in the quotient $S(q)$. But 
the simplicial structure on $N(q)$ induces the structure of a
$\Delta$-complex on $S(q)$. The number of (ordered) $n$-simplices of $S(q)$ is 
the number of $(n+1)$-configurations of sets in $\calV_\delta(q)$ such that 
their intersection is non-empty and the first set of the configuration 
lies in a $\Gamma$-transversal of $\calV_\delta(q)$, say, in 
\[\{\{\pi_q(x)\}\times T_i^{-1}K_i\mid x\in A_i^c, i\in\{1,\ldots, m\}\}\cup 
\{\{\pi_q(x)\}\times K_i\mid x\in R_i^c, i\in\{1,\ldots, m\}\}.\] 
On the other hand, $\mult^{\calV_\delta(q)}(n+1)(y)$ with $y\in\Gamma/\Gamma_q$ is the 
number of $(n+1)$-configurations of sets in $\calV_\delta(q)$ such that their intersection is 
non-empty and lies in the component $\{y\}\times\widetilde M$ and the first set of the configuration 
lies in the above $\Gamma$-transversal. The probability measure on $\Gamma/\Gamma_q$ is the normalized 
counting measure. Hence the number of $n$-simplices in $S(q)$ is bounded by 
$[\Gamma:\Gamma_q]\expected \mult^{\calV_\delta(q)}(n+1)$ and 
with~\eqref{eq: endgame expectation} the first statement follows. 

Finally, we construct equivariant maps $\tilde f:
\Gamma/\Gamma_q\times\widetilde M\to N(q)$ and 
$\tilde g: N(q)\to\Gamma/\Gamma_q\times\widetilde M$ and an
equivariant homotopy $\tilde g\circ\tilde f\simeq\id$. 
The existence of the map $\tilde g$ and the equivariant homotopy are ultimately a 
consequence of the general fact that $\widetilde M$ as a model of the classifying space 
$E\Gamma$ is a terminal object in the homotopy category of free $\Gamma$-CW complexes. 

The maps $f$ and $g$ and the homotopy $g\circ
f\simeq\id$ in the statement of the lemma will be the
induced maps on orbit spaces. This will finish the proof, since  
\[  \Gamma\bs\bigl(\Gamma/\Gamma_q\times\widetilde M\bigr) \cong
\Gamma_q\bs\widetilde M.\]
By choosing an equivariant partition of unity subordinate to 
$\calV_\delta(q)$ one obtains an equivariant map $\tilde f: \Gamma/\Gamma_q\times \widetilde M\to N(q)$, called nerve map. See~\cite{bredon}*{p.~133} for a construction of the nerve map. 

We construct the map $\tilde g$ by an induction over the skeleta of $N(q)$. To this end, 
one chooses for every set $V$ in
$\calV_\delta(q)$ a point $m_V\in V$ in an equivariant way. We define 
$\tilde g$ on the $0$-skeleton by mapping the vertex associated to
$V$ to~$m_V$. 
The $i$-skeleton $N(q)^{(i)}$ of $N(q)$, $i\ge 1$, is built from
  the $(i-1)$-skeleton by attaching equivariant $i$-cells $\Gamma\times
  D^i$ along equivariant attaching maps from $\Gamma\times S^{i-1}$ to the $(i-1)$-skeleton. 
  First let $i=1$. If two subsets of $\calV_\delta(q)$ intersect, they
    lie in the same path component of $\Gamma/\Gamma_q\times \widetilde M$. 
	Thus, if $\phi: \Gamma\times S^0\to N(q)$ is the attaching map of an equivariant $1$-cell, 
	then \[\{e\}\times S^0\xrightarrow{\phi\vert_{\{e\}\times S^0}} N(q)^{(0)}\xrightarrow{\tilde g}\Gamma/\Gamma_q\times \widetilde M\] 
	can be extended to 
	$\{e\}\times D^1$, and the latter has a unique equivariant extension to $\Gamma\times D^1$. 
    Next let $i\ge 2$. Since the 
	homotopy group  $\pi_{i-1}(\Gamma/\Gamma_q\times \widetilde M)$ 
	vanishes with respect to arbitrary base points, any composition 
	\[\{e\}\times S^{i-1}\to N(q)^{(i-1)}\xrightarrow{\tilde g}\Gamma/\Gamma_q\times\widetilde M\]
	can be extended to 
	$\{e\}\times D^i$, and, as before, the latter has a unique equivariant extension 
	to $\Gamma\times D^i$. 
	
	Similarly, 
	 the equivariant homotopy between $\tilde g\circ\tilde f$ and $\id$ is 
	 constructed by an induction over the skeleta 
	 using that $\tilde f\circ \tilde g(z)$ and $z$ lie in the same path component for every
$z\in \Gamma/\Gamma_q\times\widetilde M$ and  
	 the vanishing of homotopy groups of 
	 $\Gamma/\Gamma_q\times \widetilde M$ in degrees $\ge 1$. 
\end{proof}

\begin{proof}[End of the proof of Theorem~\ref{thm: amenable covers}]
According to Subsection~\ref{sub: reduction} we may and will assume that 
$M$ is oriented. 
Let $\epsilon>0$. By Theorem~\ref{thm: smallness of multiplicity}
there is $\delta>0$ such that \[\expected\mult^{\calV_\delta}(n+1)<\epsilon/2.\]
For this $\delta$ we take $q_0$ and $f,g$ and $S(q)$ as in the preceding lemma. For $q\ge q_0$, the number 
of $n$-simplices of $S(q)$ is at most $[\Gamma:\Gamma_q]\epsilon$. 
By the isomorphism between simplicial and singular homology for
$\Delta$-complexes~\cite{hatcher}*{Theorem~2.27 on p.~128} the
homology class 
$H_n(f)([\Gamma_q\bs\widetilde M])$ has a representative that is
an integral 
linear combination of at most $[\Gamma:\Gamma_q]\epsilon$-many singular $n$-simplices. 
Hence the fundamental class 
\[[\Gamma_q\bs\widetilde M]=H_n(g)\circ
H_n(f)([\Gamma_q\bs \widetilde M])\] 
of $\Gamma_q\bs \widetilde M$ can also 
be written as an integral  linear combination of at most $[\Gamma:\Gamma_q]\epsilon$-many singular $n$-simplices. 
Since $\epsilon>0$ was arbitrary, Theorem~\ref{thm: amenable covers} finally follows 
from Theorems~\ref{thm: estimate torsion homology from integral simplicial volume} 
and~\ref{thm: estimate rank homology from integral simplicial volume}. 
\end{proof}


\begin{bibdiv}
\begin{biblist}



\bib{abert+nikolov}{article}{
   author={Ab{\'e}rt, Mikl{\'o}s},
   author={Nikolov, Nikolay},
   title={Rank gradient, cost of groups and the rank versus Heegaard genus
   problem},
   journal={J. Eur. Math. Soc. (JEMS)},
   volume={14},
   date={2012},
   number={5},
   pages={1657--1677},
}

\bib{ballmann+gromov+schroeder}{book}{
   author={Ballmann, Werner},
   author={Gromov, Mikhael},
   author={Schroeder, Viktor},
   title={Manifolds of nonpositive curvature},
   series={Progress in Mathematics},
   volume={61},
   publisher={Birkh\"auser Boston Inc.},
   place={Boston, MA},
   date={1985},
}




\bib{bergeron+venkatesh}{article}{
   author={Bergeron, Nicolas},
   author={Venkatesh, Akshay},
   title={The asymptotic growth of torsion homology for arithmetic groups},
   journal={J. Inst. Math. Jussieu},
   volume={12},
   date={2013},
   number={2},
   pages={391--447},
}
% \bib{bergeron-cycle}{article}{
%     author = {Bergeron, Nicolas}
% 	author = {Haluk Sengun, Mehmet},
% 	author = {Venkatesh, Akshay},
%     title  = {Torsion homology growth and cycle complexity of arithmetic manifolds},
%     eprint = {arXiv: 1401.6989},
% 	date   = {2014},
% }

\bib{bredon}{book}{
   author={Bredon, Glen E.},
   title={Introduction to compact transformation groups},
   note={Pure and Applied Mathematics, Vol. 46},
   publisher={Academic Press},
   place={New York},
   date={1972},
}

\bib{curtis+reiner}{book}{
   author={Curtis, Charles W.},
   author={Reiner, Irving},
   title={Methods of representation theory. Vol. I},
   note={With applications to finite groups and orders;
   Pure and Applied Mathematics;
   A Wiley-Interscience Publication},
   publisher={John Wiley \& Sons Inc.},
   place={New York},
   date={1981},
   %pages={xxi+819},
   %isbn={0-471-18994-4},
   %review={\MR{632548 (82i:20001)}},
}
\bib{emery}{article}{
   author={Emery, Vincent},
   title={Torsion homology of arithmetic lattices and $K_2$ of imaginary
   fields},
   journal={Math. Z.},
   volume={277},
   date={2014},
   number={3-4},
   pages={1155--1164},
   %issn={0025-5874},
   %review={\MR{3229985}},
   %doi={10.1007/s00209-014-1298-2},
}

\bib{gelander}{article}{
   author={Gelander, Tsachik},
   title={Homotopy type and volume of locally symmetric manifolds},
   journal={Duke Math. J.},
   volume={124},
   date={2004},
   number={3},
   pages={459--515},
}

\bib{gromov-book}{book}{
   author={Gromov, Misha},
   title={Metric structures for Riemannian and non-Riemannian spaces},
   series={Modern Birkh\"auser Classics},
   note={Based on the 1981 French original;
   With appendices by M. Katz, P. Pansu and S. Semmes;
   Translated from the French by Sean Michael Bates},
   publisher={Birkh\"auser Boston Inc.},
   place={Boston, MA},
}
% \bib{gromov-filling}{article}{
%    author={Gromov, Mikhael},
%    title={Filling Riemannian manifolds},
%    journal={J. Differential Geom.},
%    volume={18},
%    date={1983},
%    number={1},
%    pages={1--147},
% }
\bib{gromov}{article}{
   author={Gromov, M.},
   title={Volume and bounded cohomology},
   journal={Inst. Hautes \'Etudes Sci. Publ. Math.},
   number={56},
   date={1982},
   pages={5--99 (1983)},
}
\bib{gromov-large}{article}{
   author={Gromov, M.},
   title={Large Riemannian manifolds},
   conference={
      title={Curvature and topology of Riemannian manifolds},
      address={Katata},
      date={1985},
   },
   book={
      series={Lecture Notes in Math.},
      volume={1201},
      publisher={Springer},
      place={Berlin},
   },
   date={1986},
   pages={108--121},
   %review={\MR{859578 (87k:53091)}},
   %doi={10.1007/BFb0075649},
}
		
\bib{guth}{article}{
   author={Guth, Larry},
   title={Volumes of balls in large Riemannian manifolds},
   journal={Ann. of Math. (2)},
   volume={173},
   date={2011},
   number={1},
   pages={51--76},
}


\bib{hatcher}{book}{
   author={Hatcher, Allen},
   title={Algebraic topology},
   publisher={Cambridge University Press},
   place={Cambridge},
   date={2002},
   pages={xii+544},
  }
\bib{lackenby-tau}{article}{
   author={Lackenby, Marc},
   title={Large groups, property $(\tau)$ and the homology growth of
   subgroups},
   journal={Math. Proc. Cambridge Philos. Soc.},
   volume={146},
   date={2009},
   number={3},
   pages={625--648},
}

\bib{lackenby-large}{article}{
   author={Lackenby, Marc},
   title={Detecting large groups},
   journal={J. Algebra},
   volume={324},
   date={2010},
   number={10},
   pages={2636--2657},
}



% \bib{lueck-approximation}{article}{
%    author={L{\"u}ck, W.},
%    title={Approximating $L^2$-invariants by their finite-dimensional
%    analogues},
%    journal={Geom. Funct. Anal.},
%    volume={4},
%    date={1994},
%    number={4},
%    pages={455--481},
% }

\bib{lueck-homology}{article}{
   author={L{\"u}ck, W.},
   title={Approximating $L\sp 2$-invariants and homology growth},
   journal={Geom. Funct. Anal.},
   volume={23},
   date={2013},
   number={2},
   pages={622--663},
}

\bib{lueck}{book}{
   author={L{\"u}ck, Wolfgang},
   title={$L^2$-invariants: theory and applications to geometry and
   $K$-theory},
   series={Ergebnisse der Mathematik und ihrer Grenzgebiete. 3. Folge.},
   volume={44},
   publisher={Springer-Verlag},
   place={Berlin},
   date={2002},
   pages={xvi+595},
}


\bib{ornstein+weiss}{article}{
   author={Ornstein, Donald S.},
   author={Weiss, Benjamin},
   title={Ergodic theory of amenable group actions. I. The Rohlin lemma},
   journal={Bull. Amer. Math. Soc. (N.S.)},
   volume={2},
   date={1980},
   number={1},
   pages={161--164},
}


\bib{sauer-volume}{article}{
   author={Sauer, Roman},
   title={Amenable covers, volume and $L\sp 2$-Betti numbers of aspherical
   manifolds},
   journal={J. Reine Angew. Math.},
   volume={636},
   date={2009},
   pages={47--92},
}




\bib{soule}{article}{
   author={Soul{\'e}, C.},
   title={Perfect forms and the Vandiver conjecture},
   journal={J. Reine Angew. Math.},
   volume={517},
   date={1999},
   pages={209--221},
}

\bib{dieck}{book}{
   author={tom Dieck, Tammo},
   title={Transformation groups},
   series={de Gruyter Studies in Mathematics},
   volume={8},
   publisher={Walter de Gruyter \& Co.},
   place={Berlin},
   date={1987},
   pages={x+312},
   %isbn={3-11-009745-1},
   %review={\MR{889050 (89c:57048)}},
   %doi={10.1515/9783110858372.312},
}


\bib{weiss}{article}{
   author={Weiss, Benjamin},
   title={Monotileable amenable groups},
   conference={
      title={Topology, ergodic theory, real algebraic geometry},
   },
   book={
      series={Amer. Math. Soc. Transl. Ser. 2},
      volume={202},
      publisher={Amer. Math. Soc.},
      place={Providence, RI},
   },
   date={2001},
   pages={257--262},
}


\end{biblist}
\end{bibdiv}
\end{document}